\documentclass[12pt]{amsart}
\usepackage{amscd, amssymb,latexsym,amsmath, amscd, amsmath}
\usepackage{mathrsfs}
\usepackage[all]{xy}
\usepackage{anysize}
\usepackage[sc]{mathpazo} 
\usepackage{color}
\marginsize{2.5cm}{2.5cm}{2.5cm}{2.5cm}
\newtheorem{theorem}{Theorem}[section] %(If you want theorem numbered
\newtheorem{lemma}[theorem]{Lemma}%               with section number.  Same
\newtheorem{corollary}[theorem]{Corollary}%       goes for lemmas, etc.)
\newtheorem{proposition}[theorem]{Proposition}
\theoremstyle{definition}

\newtheorem{definition}{Definition}
\newtheorem*{maintheorem}{Main Theorem }
\theoremstyle{remark}
\newtheorem{remark}{Remark}

\begin{document}

\title[Hecke algebras of $p$-adic $\mathrm{GL}_n(\mathcal{D})$]{Smooth Representations and Hecke algebras of $p$-adic $\mathrm{GL}_n(\mathcal{D})$}
\author[Mondal \& Pattanayak]{Amiya Kumar Mondal and Basudev Pattanayak}

\address{Department of Mathematics, Indian Institute of Science Education and Research Berhampur, Odisha-760010 India.}
\email{amiya@iiserbpr.ac.in}

\address{Department of Mathematics, Indian Institute of Technology Bombay, Mumbai - 400076, India.}
\email{pbasudev93@gmail.com}


\date{}

\begin{abstract}
The main question we are going to address in this paper is: How much does the representation theory of the $p$-adic group $\mathrm{GL}_n(\mathcal{D})$ depend on the $p$-adic division algebra $\mathcal{D}$? Let $\mathcal{D}$ be a central division algebra defined over some locally compact non-archimedean local field. Using Bushnell-Kutzko theory of types and S\'echerre-Stevens decomposition of spherical Hecke algebras associated to types, we obtain that the cuspidal blocks in the Bernstein decomposition of the category $\mathcal{R} \left( \mathrm{GL}_n(\mathcal{D}) \right)$ of smooth complex representations of $\mathrm{GL}_n(\mathcal{D})$ do not depend on the $p$-adic division algebra $\mathcal{D}$. In particular, when $n=1$ or $2$, the category $\mathcal{R} \left( \mathrm{GL}_n(\mathcal{D}) \right)$ does not depend on the $p$-adic division algebra $\mathcal{D}$. 
\end{abstract}

\maketitle

%-------------------------------------------Start Introduction----------------------------------------------------------
\section{Introduction}\label{intro}
Let $F$ be a locally compact non-archimedean local field with the finite residue cardinality $q_F$, and let $\mathcal{D}$ be a finite-dimensional central division algebra over the center $F$. For $N \geq 1$, let $\mathrm{GL}_N$ denote the $N$-dimensional general linear algebraic group, and $\mathrm{GL}_N(\mathcal{D})$ be the group of $F$-rational points of an inner form of $\mathrm{GL}_{N \cdot d}$, where $d^2$ denotes the dimension of $\mathcal{D}$ over $F$. One can think of $\mathrm{GL}_N(\mathcal{D})$ as the group of automorphisms of a right $\mathcal{D}$-vector space $\mathscr{V}$ of dimension $N$. Let $C_c^{\infty}(\mathrm{GL}_N(\mathcal{D}),\mathbb{C})$ be the space of locally constant complex valued functions $\phi: \mathrm{GL}_N(\mathcal{D}) \rightarrow \mathbb{C}$ with compact support. Let us choose a Haar measure $dg$ on $\mathrm{GL}_N(\mathcal{D})$ and that will determine a convolution product $*$ on the space $C_c^{\infty}(\mathrm{GL}_N(\mathcal{D}),\mathbb{C})$. The product gives $C_c^{\infty}(\mathrm{GL}_N(\mathcal{D}),\mathbb{C})$ the structure of an associative $\mathbb{C}$-algebra, known as the Hecke algebra $\mathcal{H}\left(\mathrm{GL}_N(\mathcal{D})\right)$ of $\mathrm{GL}_N(\mathcal{D})$ with respect to $dg$.

The study of the Hecke algebra and its representations is immensely connected to the theory of smooth representations of the $p$-adic groups. In particular, if $(\pi,V)$ is a smooth complex representation of $\mathrm{GL}_N(\mathcal{D})$, then $V$ becomes a non-degenerate $\mathcal{H}\left(\mathrm{GL}_N(\mathcal{D})\right)$-module and conversely, any non-degenerate $\mathcal{H}\left(\mathrm{GL}_N(\mathcal{D})\right)$-module is associated to a unique smooth representation of $\mathrm{GL}_N(\mathcal{D})$. Most fundamentally, the category $\mathcal{R}(\mathrm{GL}_N(\mathcal{D}))$ of all smooth representations of $\mathrm{GL}_N(\mathcal{D})$ on complex vector spaces is  equivalent to the category $\mathcal{H}\left(\mathrm{GL}_N(\mathcal{D})\right)$-Mods of non-degenerate $\mathcal{H}\left(\mathrm{GL}_N(\mathcal{D})\right)$-modules. 

The smooth representations of $\mathrm{GL}_N(\mathcal{D})$ are classified in \cite{DKV,Tad90}, and the Hecke algebras $\mathcal{H}\left(\mathrm{GL}_N(\mathcal{D})\right)$ have been studied in \cite{Sec09, SS08, SS12}. In this paper, our main goal is to study the Hecke algebras $\mathcal{H}\left(\mathrm{GL}_N(\mathcal{D})\right)$ explicitly for $N=1,2$. Using the one-to-one correspondences between smooth representations of $\mathrm{GL}_N(\mathcal{D})$ and non-degenerate modules of $\mathcal{H}\left(\mathrm{GL}_N(\mathcal{D})\right)$, we show that the cuspidal blocks in the Bernstein decomposition of the category $\mathcal{R} \left( \mathrm{GL}_N(\mathcal{D}) \right)$ of smooth complex representations of $\mathrm{GL}_N(\mathcal{D})$ do not depend on the $p$-adic division algebra $\mathcal{D}$ (see Theorem \ref{theorem: cusp}). In particular, we have the following interesting result of this paper: (see Corollary \ref{main:theorem:n=1} and Corollary \ref{main:corollary})
\begin{maintheorem}
 The category $\mathcal{R}\left(\mathrm{GL}_N(\mathcal{D})\right)$ of smooth complex representations of $\mathrm{GL}_N(\mathcal{D})$ does not depend on the $p$-adic division algebra $\mathcal{D}$ for $N=1,2$.   
\end{maintheorem}
These results over $p$-adic division algebra are motivated by earlier works, namely \cite{Kar16, CK17} in the context of the local and global fields. To prove this result, we first use Bernstein decomposition of the algebra $\mathcal{H}\left(\mathrm{GL}_N(\mathcal{D})\right)$ as a direct sum of two-sided ideals $\mathcal{H}\left(\mathrm{GL}_N(\mathcal{D})\right)^\mathfrak{s}$  attached to the inertial equivalence class $\mathfrak{s}$ running over the Bernstein spectrum $\mathcal{B} \left(\mathrm{GL}_N(\mathcal{D}) \right)$, and then Bushnell-Kutzko's theory of types (see \cite{BK98}) implies each Bernstein component $\mathcal{H}\left(\mathrm{GL}_N(\mathcal{D})\right)^\mathfrak{s}$ is Morita equivalent to some spherical Hecke algebra associated to a type in $\mathrm{GL}_N(\mathcal{D})$. Later, we use S\'echerre-Stevens decomposition (see \cite{SS12}) of spherical Hecke algebra as a direct sum of some finitely generated complex polynomial algebras. For $N=1$ and $2$, it implies that up to Morita equivalence, the Hecke algebra $\mathcal{H}\left(\mathrm{GL}_N(\mathcal{D})\right)$ does not depend on the central division algebra $\mathcal{D}$ defined over any locally compact non-archimedean local field.

\subsection*{Acknowledgments}
The authors would like to thank Dipendra Prasad for his valuable suggestions and for pointing out a mistake after the first draft. The authors are thankful to Sandeep Repaka as well as Alan Roche for calling their attention to these problems.

%-%----------------------------------------------Preliminaries---------------------------------------------------------
\section{Preliminaries}\label{prelim}
We assume that $F$ is a locally compact non-archimedean local field, and let $\mathcal{D}$ be a finite-dimensional central division algebra over its center $F$ with dimension $\mathrm{dim}_{F}(\mathcal{D})=d^2$. The ring of integers of $F$ and $\mathcal{D}$ will be denoted by $\mathcal{O}_F$ and $\mathcal{O}_\mathcal{D}$ respectively, with their corresponding unique maximal ideals by $\mathcal{P}_F$ and $\mathcal{P}_\mathcal{D}$. Let $q_F$ be the cardinality of the  residue field $\mathcal{O}_F/\mathcal{P}_F$ and so the cardinality of $\mathcal{O}_\mathcal{D} / \mathcal{P}_\mathcal{D}$ is $q_F^d$.
Let $\varpi_F$ be the uniformizer in $F$, and $\varpi$ be the uniformizer in $\mathcal{D}$ so that $\varpi^d=\varpi_F$. The absolute valuation $|\cdot|_F$ on $F$ is normalized such that $|\varpi_F|_F = q_F^{-1}$. Suppose $\mathrm{Nrd}_{\mathcal{D}/F}$  is the reduced norm map from $\mathcal{D}$ to $F$. Then, the absolute valuation $|\cdot|$ on $\mathcal{D}^\times$ defined by $|\alpha|=|\mathrm{Nrd}_{\mathcal{D}/F}(\alpha)|_F^d$ for $\alpha \in \mathcal{D}^\times$ is normalized so that $|\varpi| = q_F^{-d}$. Let $M_N(\mathcal{D})$ be the $F$-algebra of all $N \times N$ matrices with coefficients in $\mathcal{D}$ possessing the natural topology structure. Let $G=\mathrm{GL}_N(\mathcal{D})=M_N(\mathcal{D})^*$ be the unit group of $M_N(\mathcal{D})$. Then, the topological group $G=\mathrm{GL}_N(\mathcal{D})$ is the group of $F$-rational points of an inner form of the linear algebraic group $\mathrm{GL}_{N^\prime}$ for $N^\prime =d \cdot N$ i.e., $G$ is an inner form of $\mathrm{GL}_{N^\prime}(F)$. The classification of irreducible smooth (complex) representations of $G$ is well-known (cf. \cite{Tad90}). We denote $\mathcal{R}(G)$ for the category of all smooth representations of $G$ on complex vector spaces. 
\subsection{Hecke algebra}
Let $C_c^{\infty}(\mathrm{GL}_N(\mathcal{D}),\mathbb{C})$ denote the space of locally constant complex valued functions $\varphi: \mathrm{GL}_N(\mathcal{D}) \rightarrow \mathbb{C}$ with compact support. Let us choose a Haar measure $dg$ on $\mathrm{GL}_N(\mathcal{D})$ and that determines a convolution product $*$ on the space $C_c^{\infty}(\mathrm{GL}_N(\mathcal{D}),\mathbb{C})$ as: given functions $\varphi_1, \varphi_2 \in C_c^{\infty}(\mathrm{GL}_N(\mathcal{D}),\mathbb{C})$, one can define a function $\varphi_1 * \varphi_2:\mathrm{GL}_N(\mathcal{D}) \rightarrow \mathbb{C}$ by 
\begin{equation}\label{conv}
      \left(\varphi_1 \ast \varphi_2 \right) ( h ) = \int\limits_{\mathrm{GL}_N(\mathcal{D})} \varphi_1(hg^{-1})\varphi_2(g)~dg \quad \text{ for } h \in \mathrm{GL}_N(\mathcal{D}).
\end{equation}
This convolution product gives $\left(C_c^{\infty}(\mathrm{GL}_N(\mathcal{D}),\mathbb{C}), * \right)$ a structure of an associated $\mathbb{C}$-algebra (in general with no unit) with idempotent and is known as the Hecke algebra $\mathcal{H}\left(\mathrm{GL}_N(\mathcal{D})\right)$ of $\mathrm{GL}_N(\mathcal{D})$ with respect to $dg$.

Let $(\pi,V)$ be a smooth complex representation of $\mathrm{GL}_N(\mathcal{D})$. For each $\varphi \in \mathcal{H}\left(\mathrm{GL}_N(\mathcal{D})\right)$, we have the operator $\pi(\varphi): V \rightarrow V$ given by 
\begin{equation}\label{operator}
      \pi(\varphi) v = \int\limits_{\mathrm{GL}_N(\mathcal{D})} \varphi(g)  \pi(g)v~dg \quad \text{ for } v \in V,
\end{equation}
and it satisfies $\pi(\varphi_1) \pi(\varphi_2)= \pi(\varphi_1 * \varphi_2)$ for $\varphi_1, \varphi_2 \in C_c^{\infty}(\mathrm{GL}_N(\mathcal{D}),\mathbb{C})$. Therefore, $V$ becomes an $\mathcal{H}\left(\mathrm{GL}_N(\mathcal{D})\right)$-module and the smooth property of the representation $\pi$ determines $V$ as a non-degenerate $\mathcal{H}\left(\mathrm{GL}_N(\mathcal{D})\right)$-module that is, for each $v \in V$,  there is an element $\varphi \in \mathcal{H}\left(\mathrm{GL}_N(\mathcal{D})\right)$ such that $\pi(\varphi)v = v$. Alternatively, $V$ is non-degenerate means  $\mathcal{H}\left(\mathrm{GL}_N(\mathcal{D})\right) \cdot V=V$. Conversely any non-degenerate $\mathcal{H}\left(\mathrm{GL}_N(\mathcal{D})\right)$-module is associated to a unique smooth representation of $\mathrm{GL}_N(\mathcal{D})$. In general, the category $\mathcal{R}(\mathrm{GL}_N(\mathcal{D}))$ of all smooth representations of $\mathrm{GL}_N(\mathcal{D})$ on complex vector spaces is  equivalent to the category $\mathcal{H}\left(\mathrm{GL}_N(\mathcal{D})\right)$-Mods of non-degenerate $\mathcal{H}\left(\mathrm{GL}_N(\mathcal{D})\right)$-modules and we write it as
\begin{equation}\label{Eq:categorical}
 \mathcal{R}(\mathrm{GL}_N(\mathcal{D})) \simeq  \mathcal{H}\left(\mathrm{GL}_N(\mathcal{D})\right) \text{-Mods}.
\end{equation}

\subsection{Affine Hecke algebra} We recall the definition of the affine Hecke algebra $\mathcal{H}(r,\mathfrak{z})$, for  positive integer $r$ and for $\mathfrak z \in \mathbb{C}^\times$. The affine algebra $\mathcal{H}(r,\mathfrak{z})$ is an associative $\mathbb{C}$-algebra with identity, and this algebra is generated by the elements $s_i$ for $1 \leq i \leq r-1$ and a pair of elements $t,t^{-1}$ satisfying the following relations: (cf. \cite[Definition 5.4.6]{BK93})
\begin{itemize}
    \item[(R0)] $t \cdot t^{-1}= 1 = t^{-1} \cdot t$ 
    \vspace{1.5 mm}
    \item[(R1)] $(s_i + 1)\cdot (s_i - \mathfrak z)=0$  for   $1 \leq i \leq r-1;$
    \vspace{1.5 mm}
    \item[(R2)] $t^2 \cdot s_1=s_{n-1} \cdot t^2 ;$
    \vspace{1.5 mm}
    \item[(R3)] $t\cdot s_i = s_{i-1} \cdot t$   for  $2 \leq i \leq r-1;$
    \vspace{1.5 mm}
    \item[(R4)] $s_i \cdot s_{i+1} \cdot s_i = s_{i+1} \cdot s_i \cdot s_{i+1}$  for  $1 \leq i \leq r-2;$
    \vspace{1.5 mm}
    \item[(R5)] $s_i \cdot s_j = s_{j} \cdot s_i$  for   $1 \leq i,j \leq r-1$ with $|i-j| \geq 2$.
\end{itemize}
Clearly, when $r=1$, the relations (R1), (R2), (R3), (R4), and (R5) are vacuously satisfied. Therefore, for any parameter $\mathfrak z \in \mathbb{C}^\times$ the affine Hecke algebra
\begin{equation}\label{n:equal:1}
   \mathcal{H}(1, \mathfrak z) \cong \mathbb{C}[t,t^{-1}].
\end{equation}
When $r = 2$, the above relations (R3), (R4), and (R5) are vacuous. If the parameter $\mathfrak z = 1$, the affine Hecke algebra $\mathcal{H}(2,1)$ is the associated unital $\mathbb{C}$-algebra generated by $s,t,t^{-1}$ subject to the relations:
    \begin{itemize}
    \item[(i)] $t \cdot t^{-1}= 1 = t^{-1} \cdot t$,
    \vspace{1.5 mm}
    \item[(ii)] $(s^2 - 1)=0$ ,
    \vspace{1.5 mm}
    \item[(iii)] $t^2 s-s t^2=0 .$
\end{itemize}
We denote this associative $\mathbb{C}$-algebra $\mathcal{H}(2,1)$ by $\widetilde{\mathbb{C}} [s,t,t^{-1}] / \left\langle   s^2-1, t^2 s-s t^2 \right\rangle.$ Note that this algebra is not commutative in general. We now provide the following result on the affine Hecke algebra $\mathcal{H}(2,\mathfrak z)$, which plays a crucial role in the proof of our main result. 
\begin{lemma}\label{main:lemma}
 If $\mathfrak z + 1 \neq 0$, there is an algebra isomorphism %of the affine Hecke algebra of type $\tilde{A}_1$
   \begin{equation}
      \mathcal{H}(2,\mathfrak z)  \cong \widetilde{\mathbb{C}} [s,t,t^{-1}] / \left\langle   s^2-1, t^2 s-s t^2 \right\rangle .
   \end{equation}
\end{lemma}
\begin{proof}
Let $\mathfrak z \in \mathbb{C}^\times$ with $\mathfrak z + 1 \neq 0$. Put $s^\prime:=\frac{\mathfrak z + 1}{2} s + \frac{\mathfrak z - 1}{2}$. It is easy to check that $s^\prime, t,$ and $ t^{-1}$ generate the algebra $\mathcal{H}(2,\mathfrak z)$ under the above relations (R0), (R1), (R2), and there is a $\mathbb{C}$-algebra isomorphism $ \mathcal{H}(2,1) \xrightarrow{\sim}  \mathcal{H}(2,\mathfrak z)$ defined on generators by
\[s \mapsto \frac{\mathfrak z + 1}{2} s + \frac{\mathfrak z - 1}{2}, \text{ and } t \mapsto t~ ( \text{so }t^{-1} \mapsto t^{-1}).\] 
\end{proof}

\section{Bernstein decomposition and Hecke algebra}\label{BD}
Let $G=\mathrm{GL}_N(\mathcal{D})$ for $N \geq 1$ where $\mathcal{D}$ is a central division algebra over a non-archimedean local field $F$ and the dimension of $\mathcal{D}$ over $F$ is $d^2$. For each Levi subgroup $M$ of $G$, we denote the corresponding complex torus by $\mathrm{X}_{ur}(M)$ which is the group generated by the \textit{unramified quasicharacters} of $M$: which are smooth homomorphisms $\chi: M \rightarrow \mathbb{C}^\times$ such that $\chi$ is trivial on all compact subgroups of $M$.

We consider a cuspidal pair $(M,\sigma)$, where $M$ is a Levi subgroup of $G= \mathrm{GL}_N(\mathcal{D})$ and $\sigma$ is an irreducible supercuspidal representation  of $M$. Define an equivalence relation called \textit{inertial equivalence} on the set of all cuspidal pairs $(M,\sigma)$ as: two cuspidal pairs $(M_1,\sigma_1)$ and $(M_2,\sigma_2)$ are called \textit{inertially equivalent} if there exist $g \in G$ and $\chi \in \mathrm{X}_{ur}(M_2)$ such that 
$(M_2,\sigma_2 \otimes \chi)= ({^g}M_1,{^g}\sigma_1)$, where ${^g}\sigma_1$ is defined by ${^g}\sigma_1(x)=\sigma_1(gxg^{-1})$, for $x \in  {^g}M_1= g^{-1} M_1g$. Let $[M,\sigma]_{G}$ be the inertial equivalence class corresponding to the cuspidal pair $(M,\sigma)$. Denote $\mathcal{B}(G)$ for the Bernstein spectrum of $G$, which is the set of all inertial equivalence classes in $G$. We say that a smooth irreducible representation $(\pi, V)$ has \textit{inertial support} $[M,\sigma]_{G}$  if $(\pi, V)$ appears as a subquotient of a representation parabolically induced from some element of $[M,\sigma]_{G}$. 
\subsection{Bernstein decomposition}
Let $\mathcal{R}(G)$ be the category of all smooth complex representations of $G$. For each $\mathfrak{s} :=[M,\sigma]_{G} \in \mathcal{B}(G)$, one can define a full subcategory $\mathcal{R}^{\mathfrak{s}}(G)$ of $\mathcal{R}(G)$ consisting of those smooth complex representations $(\pi,V) \in \mathcal{R}(G)$ whose each irreducible subquotient has inertial support $\mathfrak{s}=[M,\sigma]_{G}$. the subcategory $\mathcal{R}^{\mathfrak{s}}(G)$ is called the \textit{Bernstein block} of $\mathcal{R}(G)$ corresponding to $\mathfrak{s}=[M,\sigma]_{G}$ .

\begin{theorem}\cite[Theorem 2.10]{Ber}\label{Bernstein}
	The Bernstein decomposition gives a direct product decomposition of $\mathcal{R}(G)$ into indecomposable subcategories $\mathcal{R}^{\mathfrak{s}}(G)$: 
	\[\mathcal{R}(G) = \prod\limits_{\mathfrak{s} \in \mathcal{B}(G)}\mathcal{R}^{\mathfrak{s}}(G),\] where $\mathfrak{s}$ runs over the spectrum $\mathcal{B}(G)$. Concretely, if $(\pi,V) \in \mathcal{R}(G)$, then for each $\mathfrak{s} \in \mathcal{B}(G)$, $V$ has a unique maximal $G$-subspace $V_{\mathfrak{s}} \in \mathcal{R}^{\mathfrak{s}}(G)$ such that \[ V= \bigoplus\limits_{\mathfrak{s} \in \mathcal{B}(G)} V_{\mathfrak{s}}.\]
\end{theorem}
 Moreover, if $\mathfrak{s},\mathfrak{s}^\prime \in \mathcal{B}(G)$ with $\mathfrak{s}  \neq \mathfrak{s}^\prime$, then $\mathrm{Hom}_{G}(\mathcal{R}^{\mathfrak{s}}(G),\mathcal{R}^{\mathfrak{s}^\prime}(G))=0$.

\subsection{Bernstein decomposition of Hecke algebra} For each equivalence class $\mathfrak s \in \mathcal{B}(G)$, let $\mathcal{H}\left(G\right)^\mathfrak{s}$ be the two-sided ideal of $\mathcal{H}\left(G\right)$ corresponding to all smooth representations $(\pi,V)$ of $G$ of inertial support $\mathfrak{s}$. In particular, $\mathcal{H}\left(G\right)^\mathfrak{s}$ is the unique $G$-subspace of $\mathcal{H}\left(G\right)$ lying in $\mathcal{R}^\mathfrak{s}(G)$ such that it is maximal for this property. Each ideal $\mathcal{H}(G )^{\mathfrak{s}}$ is a non-commutative, non-unital, non-finitely-generated,
non-reduced $\mathbb{C}$-algebra and it is characterized by $\mathcal{H}(G )^{\mathfrak{s}} \cdot V= V$ for all $V \in \mathcal{R}^{\mathfrak{s}}(G)$ (see \cite{ABP, Sol} for more details).

The Bernstein decomposition (Theorem \ref{Bernstein}) induces a factorization of the Hecke algebra $\mathcal{H}(G)$ in two-sided ideals as: 
\begin{equation}\label{Hecke:decomposition}
\mathcal{H}(G) = \bigoplus_{s \in \mathcal{B}(G)} \mathcal{H}(G)^{\mathfrak{s}}.    
\end{equation}
The $\mathbb{C}$-algebra $\mathcal{H}(G)^{\mathfrak{s}}$ is called the \emph{Bernstein block} of the Hecke algebra  $\mathcal{H}(G)$ corresponding to the inertial class $\mathfrak{s}$. If $\mathcal{H}(G)^{\mathfrak{s}}$-Mods denote the category of all unitary $\mathcal{H}(G)^{\mathfrak{s}}$-modules then there is a natural equivalence of categories:
\begin{equation}\label{Eq:s-level}
    \mathcal{R}^{\mathfrak{s}}(G) \simeq \mathcal{H}(G)^{\mathfrak{s}} \text{-Mods.}
\end{equation}

\subsection{Spherical Hecke algebra}\label{SHA}
Let $(\lambda , W)$ be an irreducible smooth representation of a compact
open subgroup $K$ of $G$. The $\lambda$-\emph{spherical Hecke algebra} $\mathcal{H}(G,(K,\lambda))$ associated to the pair $(K,\lambda)$ is the vector space of all compactly
supported functions $\Phi: G \rightarrow \mathrm{End}_{\mathbb{C}}(W^\vee)$ satisfying
\begin{equation*}
	\Phi(k_1 g k_2) = \lambda^\vee(k_1) \Phi(g) \lambda^\vee(k_2) ~ \text{ for all } k_1, k_2 \in K \text{ and } g \in G,
\end{equation*}
where $(\lambda^\vee,W^\vee)$ denote the contragredient representation of $(\lambda, W)$. The standard convolution operation `$*$' defined by 
\begin{equation*} 
	\Phi_1 * \Phi_2: h \longmapsto \int\limits_{G} \Phi_1(g) \Phi_2(g^{-1}h)~dg, \text{ for } \Phi_1,\Phi_2 \in \mathcal{H}(G,(K,\lambda)),
\end{equation*} 
gives the space $\mathcal{H}(G,(K,\lambda))$ a structure of an associative $\mathbb{C}$-algebra with identity. Whenever the compact open subgroup $K$ of $G$ is explicitly mentioned, we write $\mathcal{H}(G,\lambda)$ for the $\lambda$-spherical Hecke algebra $\mathcal{H}(G,(K,\lambda))$ by dropping the notation $K$.  

If $(\pi, V)$ is a smooth complex representation of $G$, then the $\lambda$-\textit{isotypic subspace} $V^\lambda$ of $V$ is the sum of all irreducible $K$-subspaces of $V$, which are equivalent to $\lambda$. Let $\mathcal{R}_\lambda (G)$ denote the subcategory of $\mathcal{R}(G)$ whose objects are the smooth complex representations
$(\pi, V)$ of $G$ generated by the $\lambda$-isotypic subspace $V^\lambda$ of $V$. If $\mathcal{H}(G,\lambda)$-Mods denote the category of all unitary left-modules over the Hecke algebra $\mathcal{H}(G,\lambda)$, there is a functor
\begin{equation}\label{functor_M}
	\mathcal{M}_\lambda : \mathcal{R}_\lambda (G) \rightarrow \mathcal{H}(G,\lambda)\text{-Mods},
\end{equation}
defined by
\begin{equation*}
	\pi \mapsto \mathrm{Hom}_K(\lambda,\pi).
\end{equation*}
The functor $\mathcal{M}_\lambda$ plays an important role in the context of the theory of types.

\subsection{Notion of types and G-covers}
In \cite{BK98}, Bushnell and Kutzko described the abelian category $\mathcal{R}^{\mathfrak{s}}(G)$ in terms of a class of pairs $(K,\lambda)$ consisting of irreducible smooth representations $\lambda$ of compact open subgroups $K$ of $G$. The pair $(K,\lambda)$ is a \textit{type} in $G$ if the subcategory $\mathcal{R}_\lambda (G)$ is closed under subquotients in $\mathcal{R}(G)$.

For any finite subset $\mathfrak{S} \subset \mathcal{B}(G)$, denote $\mathcal{R}^{\mathfrak{S}}(G)= \prod\limits_{\mathfrak{s} \in \mathfrak{S}}\mathcal{R}^{\mathfrak{s}}(G)$ i.e., $\mathcal{R}^{\mathfrak{S}}(G)$ is the full subcategory of $\mathcal{R}(G)$ consisting of those smooth complex representations $(\pi,V) \in \mathcal{R}(G)$ whose each irreducible subquotient has \textit{inertial support} contained in $\mathfrak{S}$. Let $\lambda$ be an irreducible smooth representation of a compact open subgroup $K$ of $G$. For a finite subset $\mathfrak{S}$ of $\mathcal{B}(G)$, the pair $(K,\lambda)$ is said to be an $\mathfrak{S}$-\textit{type} in $G$ if for any irreducible smooth representation $(\pi,V)$ of $G$, the representation $(\pi,V) \in \mathcal{R}^{\mathfrak{S}}(G)$ if and only if $\pi$ contains $\lambda$ i.e. $\mathrm{Hom}_K(\lambda,\pi) \neq 0$. In \cite[Theorem (4.3)]{BK98}, Bushnell and Kutzko showed that if the pair $(K,\lambda)$ is an $\mathfrak{S}$-type in $G$ , we have
\begin{equation}\label{Eq:type}
  \mathcal{R}_{\lambda}(G)=\mathcal{R}^{\mathfrak{S}}(G) \text{ as subcategories of } \mathcal{R}(G), 
\end{equation}
and the functor $\mathcal{M}_\lambda $ as mentioned in Eq. (\ref{functor_M}) induces an equivalence of the categories:
	\begin{equation}\label{Eq:func:M}
		\mathcal{M}_\lambda : \mathcal{R}_\lambda (G) \xrightarrow{\sim} \mathcal{H}(G,\lambda)\text{-Mods}.
	\end{equation}
Note that when $\mathfrak{S}$ is a singleton set $\{\mathfrak{s}\}$ in $\mathcal{B}(G)$, we write $\mathfrak{s}$-\textit{type} instead of $\{\mathfrak{s}\}$-type.

Let $M$ be a proper Levi subgroup of $G$. Consider the pair $(K_{M},\lambda_{M})$ consisting of a compact open subgroup $K_M$ of $M$ and an irreducible smooth representation $\lambda_M$ of the subgroup $K_M$. Let $K$ be a compact open subgroup of $G$ and $\lambda$ be an irreducible smooth representation of $K$.
\begin{definition}[G-cover]
	 The pair $(K,\lambda)$ is called a $G$-\textit{cover} of $(K_M,\lambda_M)$ if  for any opposite pair of parabolic subgroups $P=MU$ and $P^-=MU^-$ with common Levi factor $M$ and unipotent radicals $U$ and $U^-$ respectively, the pair $(K,\lambda)$ satisfies the following properties: (see \cite[\S 3.3]{MP} for this version definition)
	\begin{enumerate}
		\item[(i)]  $K$ decomposes w.r.t $(U,M,U^- )$ i.e., $ K = (K \cap U)(K \cap M)(K \cap U^-).$
		\item[(ii)] $K_M = K \cap M$, $\lambda|_{K_M} =\lambda_M$ and $K \cap U, K \cap U^-$  is contained in the kernel of $\lambda$.
		\item[(iii)]  For any smooth representation $(\pi,V)$ of $G$; the natural projection $V$ to the Jacquet module $V_{U}$ induces an injection on $V^\lambda$.
	\end{enumerate}
\end{definition}
Let $L$ and $M$ are Levi subgroups of $G$ such that $L$ is contained in $M$ and the pair $(K_L,\lambda_L)$ is consisting of a compact open subgroup $K_L$ of $L$ together with an irreducible representation $\lambda_L$ of the subgroup $K_L$, then the following transitivity properties of the cover hold:
\begin{itemize}
	\item[(1)] If $(K,\lambda)$ is a $G$-{cover} of $(K_M,\lambda_M)$ and $(K_M,\lambda_M)$ is a $M$-{cover} of $(K_L,\lambda_L)$, then $(K,\lambda)$ is a $G$-{cover} of $(K_L,\lambda_L)$.
	\item[(2)] Suppose $(K,\lambda)$ is a $G$-{cover} of $(K_L,\lambda_L)$. Denote $K \cap M$ by $K_M$, and $\lambda|_{K_M}$ by $ \lambda_M$. Then $(K,\lambda)$ is a $G$-{cover} of $(K_{M},\lambda_{M})$ and $(K_M,\lambda_M)$ is a $M$-{cover} of $(K_L,\lambda_L)$.
\end{itemize}

\section{Classification of Representations of $\mathrm{GL}_2(\mathcal{D})$}\label{tadic}
In \cite{Tad90}, Marko Tadi\'c gave the classification of the irreducible smooth representations of $\mathrm{GL}_N(\mathcal{D})$ where $\mathcal{D}$ is a central division algebra over a local field $F$. Specifically, he gave an irreducibility criterion of the principal series representations of $\mathrm{GL}_N(\mathcal{D})$  by using the  Langlands classification theory when the characteristic of $F$ is $0$. We now recall a $\mathrm{GL}_2(\mathcal{D})$ case of Tadi\'c results.

Let $M_0$ be the diagonal subgroup $\mathcal{D}^\times \times \mathcal{D}^\times$ of $G=\mathrm{GL}_2(\mathcal{D})$. Let $P_0$ be the standard minimal parabolic subgroup of $G$ consisting of upper triangular matrices in $G$ with $P_0 = M_0 N_0$, where the unipotent radical $N_0$ is the group of upper triangular matrices with diagonal entries $1$. Let $\sigma_1$ and $\sigma_2$ be two irreducible smooth representations of $\mathcal{D}^\times$ and $\sigma_1 \times \sigma_2$ denote the normalized parabolically induced representation $\mathrm{Ind}^G_{P_0}(\sigma_1 \otimes \sigma_2)$ of $G$. For an irreducible representation $\sigma$ of $\mathcal{D}^\times$, let $a(\sigma)$ denote the length of the segment (say $\Delta$) which determines the essentially square-integrable representation $\mathrm{JL}(\sigma)$  (of the form $Q(\Delta)$), which is the Jacquet–Langlands lift of $\sigma$. Then, the representation $\sigma_1 \times \sigma_2$ is reducible if and only if $\sigma_2 \simeq \sigma_1 \otimes |\cdot|^{\pm \frac{a(\sigma_1)}{d}}$ (cf. \cite[Lemma 2.5 and 4.2]{Tad90}). If $\pi$ is an irreducible smooth representation of $\mathrm{GL}_2(\mathcal{D})$,  $\pi$ satisfies  one of the following four exhaustive kinds uniquely: (see \cite[Theorem 2.2]{Rag} for more details)
\begin{itemize}
    \item[(I)] $\pi$ is a purely supercuspidal representation, that is $\pi$ lies in the Bernstein block attached to the inertial class $\mathfrak s=[G,\pi]_G$.
    \item[(II)] $\pi$ is an irreducible parabolically induced representation $\sigma_1 \times \sigma_2$ that is, there are some irreducible representations $\sigma_1$, $\sigma_2$ of $\mathcal{D}^\times$ with $\sigma_2 \not\simeq \sigma_1 \otimes |\cdot|^{\pm \frac{a(\sigma_1)}{d}}$ such that $\pi = \mathrm{Ind}^G_{P_0}(\sigma_1 \otimes \sigma_2)$. This kind of representation is said to be \textit{irreducible principal-series representation}.
    \item[(III)] $\pi$ is a finite-dimensional representation, and so $\pi$ is one-dimensional of the form $\chi \circ \mathrm{Nrd}_{\mathcal{M}/F}$ for some quasi-character $\chi$ of $F^\times$. Here $\mathrm{Nrd}_{\mathcal{M}/F}$ denotes the reduced norm map from $M_2(\mathcal{D})$ to $F$. 
    \item[(IV)] $\pi$ is an infinite-dimensional irreducible quotient of a reducible parabolically induced representation ~ $\sigma_1 \times \sigma_2$. This kind of representation is called a special representation. There exists a unique irreducible (infinite-dimensional) quotient in the parabolically induced representation $\sigma \times \sigma ~ |\cdot|^{\pm \frac{a(\sigma)}{d}}$. The unique irreducible quotient $\mathrm{St}(\sigma)$ of ~ $\sigma ~ |\cdot|^{- \frac{a(\sigma)}{2d}} \times \sigma ~ |\cdot|^{ \frac{a(\sigma)}{2d}}$ is said to be a generalized \emph{Steinberg} representation and the unique irreducible quotient $\mathrm{Sp}(\sigma)$ of ~ $\sigma ~ |\cdot|^{ \frac{a(\sigma)}{2d}} \times \sigma ~ |\cdot|^{ -\frac{a(\sigma)}{2d}}$ is said to be a generalized \emph{Speh} representation of $G$. 
\end{itemize}

\section{Hecke algebras for Inner Forms}
In the series of papers \cite{Sec05, SS08, SS12}, S\'echerre and Stevens have generalized the Bushnell-Kutzko theory of types to $G=\mathrm{GL}_N(\mathcal{D})$ by constructing types for all Bernstein block $\mathcal{R}(G)^\mathfrak{s}$ corresponding to each inertial class $\mathfrak s$ in $G$ using the theory of covers and the notion of a common approximation of simple characters.

\subsection{Construction of types}\label{Construct:types} Fix an inertial equivalence class $\mathfrak{s}=[L,\sigma]_G$ in $G$. There exists a subspace decomposition $\mathscr{V}= \bigoplus\limits_{i=1}^k \mathscr{V}_i$ of the $N$-dimensional $\mathcal{D}$-vector space $\mathscr{V}$ such that $L$ is the stabilizer of the decomposition, in particular if $n_i$ denote the dimension of $\mathscr{V}_i$ over $D$, then $L= \prod\limits_{i=1}^k \mathrm{GL}_{n_i}(\mathcal{D})$. Also, there exists an irreducible cuspidal representation $\sigma_i$ of $\mathrm{GL}_{n_i}(\mathcal{D})$ for each $i \in \mathcal{I}_{\mathfrak s}=\{1, \cdots, k\}$ such that $\sigma= \bigotimes\limits_{i=1}^k \sigma_i$.
\begin{definition}
    Define an equivalence relation $\sim$ on $\mathcal{I}_{\mathfrak s}=\{1, \cdots, k\}$: for $i,j \in \mathcal{I}_{\mathfrak s}$
    \begin{equation*}
        i \sim j \text{ if and only if } n_i=n_j \text{ and } [\mathrm{GL}_{n_i}(\mathcal{D}), \sigma_i]_{\mathrm{GL}_{n_i}(\mathcal{D})} = [\mathrm{GL}_{n_j}(\mathcal{D}), \sigma_j]_{\mathrm{GL}_{n_j} (\mathcal{D})}.
    \end{equation*}
\end{definition}
We can make the assumption: $\sigma_i=\sigma_j$ if $i \sim j$, by replacing the cuspidal pair $(L,\sigma)$ with an inertially equivalent pair if necessary. Suppose there are $\ell$ equivalence classes on $\mathcal{I}_{\mathfrak s}$ denoted by $ \mathcal{I}_1,...,\mathcal{I}_\ell$. We put $\mathcal{Y}_t= \bigoplus\limits_{i \in \mathcal{I}_t} \mathscr{V}_i$. We denote $M$ for the subgroup of $G$, which stabilizes the decomposition $\mathscr{V}= \bigoplus\limits_{t=1}^\ell \mathcal{Y}_t$. Then, $M$ is a Levi subgroup of $G$ such that $M$ contains $L$.

Corresponding to the $G$-inertial equivalence class $\mathfrak{s}=[L,\sigma]_G$, we have the cuspidal $L$-inertial class $\mathfrak{s}_L=[L,\sigma]_L$. In \cite{Sec05, SS08}, they showed that there exists a maximal simple $[\mathrm{GL}_{n_i}(\mathcal{D}), \sigma_i]_{\mathrm{GL}_{n_i}(\mathcal{D})}$-type $(K_i, \lambda_i)$ (unique up to $\mathrm{GL}_{n_i}(\mathcal{D})$-conjugacy) contained in $\sigma_i$, and if we set $K_L=\prod\limits_{i=1}^k K_i$ with $\lambda_L= \bigotimes\limits_{i=1}^k \lambda_i$, the pair $(K_L, \lambda_L)$ is an $\mathfrak{s}_L$-type in $L$. Then, they constructed an $M$-cover $(K_M, \lambda_M)$ of the $\mathfrak{s}_L$-type $(K_L, \lambda_L)$ and therefore, $\mathfrak{s}_M=[L,\sigma]_M$-type in $M$. Later, they constructed a $G$-cover $(K,\lambda)$ of the type $(K_M, \lambda_M)$ in \cite[Theorem 8.2]{SS12}. Using transitivity properties of covers, it follows that $(K, \lambda)$ is an $\mathfrak{s}=[L,\sigma]_G$-type in $G$. 

\begin{lemma}\label{Morita:equivalence}
 For each inertial class $\mathfrak{s} \in \mathcal{B}(G)$, there exists an $\mathfrak{s}$-type $(K,\lambda)$ in $G$ such that the two sided ideal $\mathcal{H}(G)^\mathfrak{s}$ is Morita equivalent to the $\lambda$-spherical Hecke algebra $\mathcal{H}(G,\lambda)$.
\end{lemma}
\begin{proof}
    For $\mathfrak{s} \in \mathcal{B}(G)$, the existence of an $\mathfrak{s}$-type in $G$ follows from the work of S\'echerre and Stevens in \cite{SS12}. Let $(K,\lambda)$ be an $\mathfrak{s}$-type in $G$. Then, combining equivalence (\ref{Eq:type}) and (\ref{Eq:func:M}) we have an equivalence of categories:
    \begin{equation}
        \mathcal{R}^{\mathfrak{s}} (G) \simeq  \mathcal{H}(G,\lambda)\text{-Mods}.
    \end{equation}
 Now this equivalence together with the equivalence (\ref{Eq:s-level}) induces an equivalence of categories  
 \begin{equation*}
        \mathcal{H}(G)^{\mathfrak{s}} \text{-Mods} \simeq \mathcal{H}(G,\lambda)\text{-Mods},
    \end{equation*}
    which gives us a Morita equivalence between $\mathcal{H}(G)^{\mathfrak{s}}$ and $  \mathcal{H}(G,\lambda)$.
\end{proof}

\subsection{S\'echerre-Stevens decomposition}
We now recall the main result in \cite{SS12} dealing with the decomposition of $\mathcal{H}(G,\lambda)$ as a product of affine Hecke algebras of type $A$.

Denote $H=\mathrm{GL}_m(\mathcal{D})$ for some $m \geq 1$, and let $\tau$ be an irreducible cuspidal representation of $H$. We can attach two numerical invariants $n(\tau)$ and $s(\tau)$ associated to $\tau$. The group of unramified characters $\chi$ of $H$ such that $\tau \simeq \tau \chi$, is finite (cf. \cite[Proposition 4.1]{Sec09}) and the cardinality of this finite group is called the {\em torsion number} $n(\tau)$ of $\tau$. For the other invariant, we take $\tilde{H}  = \mathrm{GL}_{2m}(\mathcal{D})$ and let $\tilde{P}$ be the standard parabolic subgroup of $\tilde{H}$ with Levi subgroup $H \times H$. There exists a unique real number $s(\tau)$, called  the {\em reducibility number} of $\tau$ such that the normalized parabolic induced representation $\mathrm{ Ind}_{\tilde{P}}^{\tilde{H}}\, (\tau \bigotimes \tau ~ |\cdot|^{\frac{s(\tau)}{d}})$ is reducible (cf. \cite[Theorem 4.6]{Sec09}). Note that $s(\tau)=\pm a(\tau)$ as defined in Section \ref{tadic}.

For a natural number $r$, let $\mathcal{H}(r, \mathfrak{z})$ denote the affine Hecke algebra of type $\Tilde{A}_{r-1}$ with parameter $\mathfrak{z} \in \mathbb{C}^\times$. Let $\mathfrak{s}=[L,\sigma]_G \in \mathfrak B(G)$. Recall all the notations described in \S\ref{Construct:types} corresponding to the inertial class $\mathfrak{s}$. S\'echerre and Stevens showed:
\begin{theorem}\cite[Main Theorem]{SS12} \label{SS}
There exists an $\mathfrak s$-type $(K, \lambda)$ in $G$ such that
\[\mathcal H(G,\lambda)\cong \bigotimes\limits_{i=1}^\ell \mathcal H(r_i,q_F^{\mathfrak{f}_i}),\]
where $r_i$ is the cardinality of $\mathcal{I}_i$ and $\mathfrak{f}_i = n(\sigma_j) \cdot s(\sigma_j)$ for any $j \in \mathcal{I}_i$.
\end{theorem}
\begin{remark}\label{residue:remark}
     For each $\sigma_j$, there exists a finite unramified extension $\tilde{F}_j$ of $F$ such that $\mathfrak{f}_i$ is the residue degree of $\tilde{F}_j$ over $F$ (cf. \cite[\S 4.2, Theorem 4.6]{Sec09}).
\end{remark}

%-------------------------------------------------On GL(n,D)----------------------------------------------------------------

\section{Cuspidal blocks of $\mathrm{GL}_N(\mathcal{D})$}
Consider the subcategory $\mathcal{R}(\mathrm{GL}_N(\mathcal{D}))_{\mathrm{cusp}}$ of the category $\mathcal{R}(\mathrm{GL}_N(\mathcal{D}))$ of all smooth complex representations of $\mathrm{GL}_N(\mathcal{D})$ defined by
\[\mathcal{R}(\mathrm{GL}_N(\mathcal{D}))_{\mathrm{cusp}} := \prod\limits_{\mathfrak{s} \in \mathcal{B}(\mathrm{GL}_N(\mathcal{D}))_{\mathrm{cusp}}}\mathcal{R}^{\mathfrak{s}}(\mathrm{GL}_N(\mathcal{D})), \]
 where $\mathfrak{s}$ runs over the subset $\mathcal{B}(\mathrm{GL}_N(\mathcal{D}))_{\mathrm{cusp}}$ consisting of those inertial class $\mathfrak{s}=[L,\sigma]_{\mathrm{GL}_N(\mathcal{D})}$ in the Bernstein spectrum $\mathcal{B}(\mathrm{GL}_N(\mathcal{D}))$ which support the supercuspidal representations of $\mathrm{GL}_N(\mathcal{D})$ that is the Levi subgroup $L$ in the inertial class $\mathfrak{s}$ is $\mathrm{GL}_N(\mathcal{D})$.

\begin{proposition}\label{prop: cusp}
Let $(K, \lambda)$ be an $\mathfrak{s}$-type in $\mathrm{GL}_N(\mathcal{D})$ for $\mathfrak{s} \in \mathcal{B}(\mathrm{GL}_N(\mathcal{D}))_{\mathrm{cusp}}$. Then the Hecke algebra $\mathcal H(\mathrm{GL}_N(\mathcal{D}),\lambda)$ associated with the pair $(K, \lambda)$ is isomorphic to a finitely generated commutative $\mathbb{C}$-algebra. In particular,
    \begin{equation}
        \mathcal H(\mathrm{GL}_N(\mathcal{D}),\lambda) \cong  \mathbb{C}[x, x^{-1}].
    \end{equation}
\end{proposition}
\begin{proof}
 Since $(K,\lambda)$ is an $\mathfrak s$-type in $\mathrm{GL}_N(\mathcal{D})$, applying Theorem \ref{SS}, we get 
\[\mathcal H(\mathrm{GL}_N(\mathcal{D}),\lambda)\cong \bigotimes\limits_{i=1}^\ell \, \mathcal H(r_i,q_F^{\mathfrak{f}_i}), \text{ for } r_i = \# \mathcal{I}_i \text{ and } \mathfrak{f}_i = n(\sigma_j) \cdot s(\sigma_j).\]

In this situation, $\mathfrak s=[\mathrm{GL}_N(\mathcal{D}),\sigma]_{\mathrm{GL}_N(\mathcal{D})}$ for a supercuspidal representation $\sigma$ of $\mathrm{GL}_N(\mathcal{D})$. Using the above description in \S \ref{Construct:types}, we have $L=\mathrm{GL}_N(\mathcal{D}),~k=\# \mathcal{I}_{\mathfrak{s}}=1,~n_1=N,~\sigma_1=\sigma, ~\ell=1, ~ \mathcal{I}_1=\{1\},$ and $r_1=\# \mathcal{I}_1=1$. Then,  
\[\mathcal H(\mathrm{GL}_N(\mathcal{D}),\lambda)\cong \mathcal H(1,q_F^{\mathfrak{f}}),\] where $\mathfrak{f}=n(\sigma) s(\sigma)$ is a real number. Therefore, using isomorphism (\ref{n:equal:1}) we have
\[\mathcal H(\mathrm{GL}_N(\mathcal{D}),\lambda) \cong \mathbb{C}[x,x^{-1}].\]   
\end{proof}

\begin{theorem}\label{theorem: cusp}
Let $\mathcal{D}$ be a central division algebra defined over a locally compact non-archimedean local field. Then, the subcategory of all smooth complex representations of $\mathcal{R}(\mathrm{GL}_N(\mathcal{D}))_{\mathrm{cusp}}$ does not depend on the division algebra $\mathcal{D}$.
\end{theorem}
\begin{proof}
  Let  $\mathfrak s \in \mathcal{B}(\mathrm{GL}_N(\mathcal{D}))_{\mathrm{cusp}}$. Applying Lemma \ref{Morita:equivalence}, we have
\begin{align*}
    \mathcal{H}\left(\mathrm{GL}_N(\mathcal{D})\right)^\mathfrak s &\sim_M \mathcal{H}(\mathrm{GL}_N(\mathcal{D}), \lambda) \quad \text{for some } \mathfrak s \text{-type } (K, \lambda) \text{ as in Theorem } \ref{SS}\\
    &\cong 	\mathbb{C}[x, x^{-1}] \quad \quad (\text{by Proposition } \ref{prop: cusp})
\end{align*}
There are countably infinitely many inertial classes $\mathfrak s$ in each $\mathcal{B}(\mathrm{GL}_N(\mathcal{D}))_{\mathrm{cusp}}$. Therefore, we have the following Morita equivalence: 
\begin{align}
   \mathcal{H}\left(\mathrm{GL}_N(\mathcal{D})\right)_{\rm cusp} &:= \bigoplus\limits_{\mathfrak s \in \mathcal{B}(\mathrm{GL}_N(\mathcal{D}))_{\mathrm{cusp}}} \mathcal{H}\left(\mathrm{GL}_N(\mathcal{D})\right)^\mathfrak s  \nonumber\\
   &\sim_M \bigoplus_\nu  \mathbb{C}[x, x^{-1}] \label{E1}
\end{align}
 where $\nu$ runs over a countably infinite set.

We have the following equivalence of categories
\begin{align}
  \mathcal{R}(\mathrm{GL}_N(\mathcal{D}))_{\mathrm{cusp}} &\simeq \mathcal{H}\left(\mathrm{GL}_N(\mathcal{D})\right)_{\rm cusp} \text{-Mods} \label{E2} \\
  &\simeq \left(\bigoplus_{\nu \in \mathbb{Z}}  \mathbb{C}[x, x^{-1}]\right)\text{-Mods}, \label{E3}
\end{align}
where $A\text{-Mods}$ denotes the category of all unitary left $A$-modules. The equivalences (\ref{E2}) and (\ref{E3}) follow from eq.(\ref{Eq:s-level}) and eq.(\ref{E1}) respectively. Hence the subcategory $\mathcal{R}(\mathrm{GL}_N(\mathcal{D}))_{\mathrm{cusp}}$ does not depend on the division algebra $\mathcal{D}$, and so on the locally compact non-archimedean local field $F$ that is, $\mathcal{R}(\mathrm{GL}_N(F))_{\mathrm{cusp}}$ is independent of $F$.
\end{proof}

\subsection{Result on $\mathrm{GL}_1(\mathcal{D})$}
All irreducible smooth representations of $\mathrm{GL}_1(\mathcal{D})$, the multiplicative group $\mathcal{D}^\times$ of $\mathcal{D}$ are supercuspidal and each element of the Bernstein spectrum $\mathcal{B}(\mathcal{D}^\times)$ of $\mathcal{D}^\times$ supports only supercuspidal representations. Therefore, we have the following result:
\begin{corollary}\label{main:theorem:n=1}
Let $\mathcal{D}$ be a central division algebra defined over a locally compact non-archimedean local field. Then, the category $\mathcal{R}(\mathrm{GL}_1(\mathcal{D}))$ does not depend on $\mathcal{D}$, that is
\begin{equation*}
 \mathcal{R}(\mathrm{GL}_1(\mathcal{D})) \simeq \mathcal{R}(\mathrm{GL}_1(\mathcal{D}^\prime)),   
\end{equation*}
for any two central division algebras $\mathcal{D}$ and $\mathcal{D}^\prime$ defined over any locally compact non-archimedean local field.
\end{corollary}
\begin{proof}
   The proof follows from Theorem \ref{theorem: cusp}.
\end{proof}

%----------------------------------------------------Main results---------------------------------------------------------

\section{Results on $\mathrm{GL}_2(\mathcal{D})$}
In this section, we mainly concentrate on the representation theory of  $\mathrm{GL}_2(\mathcal{D})$ and the Hecke algebra $\mathcal{H}(\mathrm{GL}_2(\mathcal{D}))$. For the rest of the section, we fix $G=\mathrm{GL}_2(\mathcal{D})$.

If $L$ is a Levi subgroup of $\mathrm{GL}_2(\mathcal{D})$, up to conjugate $L$ is either $\mathrm{GL}_2(\mathcal{D})$ or $\mathcal{D}^\times \times \mathcal{D}^\times$. We now decompose the Bernstein spectrum $\mathcal{B}(G)$ of $\mathrm{GL}_2(\mathcal{D})$ into three exhaustive classes:
\begin{equation}\label{break}
    \mathcal{B}(G) = \mathcal{B}_{\mathrm{cusp}} ~ \sqcup ~ \mathcal{B}_{\mathrm{neqv}} ~ \sqcup ~ \mathcal{B}_{\mathrm{eqv}},
\end{equation}
where $\mathcal{B}_{\mathrm{cusp}}$ consists of those inertial class supporting the supercuspidal representations, $\mathcal{B}_{\mathrm{neqv}}$ consists of those inertial class $\mathfrak{s}=[L,\sigma]_G \in \mathfrak B(G)$ with $L=\mathcal{D}^\times \times \mathcal{D}^\times$, $\sigma= \sigma_1 \otimes \sigma_2$ such that $\sigma_2$ is not inertially equivalent to $\sigma_1$ in $\mathcal{D}^\times$, and $\mathcal{B}_{\mathrm{eqv}}$ consists of rest of the inertial class $\mathfrak{s}=[L,\sigma]_G $ in $ \mathfrak B(G)$ i.e., those $\mathfrak{s}=[\mathcal{D}^\times \times \mathcal{D}^\times, \sigma_1 \otimes \sigma_2]_G $ such that $\sigma_2$ is inertially equivalent to $\sigma_1$ in $\mathcal{D}^\times$.

\begin{proposition}\label{decompose:complex-algebra}
    If $(K,\lambda)$ is an $\mathfrak{s}$-type in $G$ occurring in S\'echerre-Stevens' Theorem \ref{SS}, then the $\lambda$-spherical Hecke algebra $\mathcal H(G,\lambda)$ is isomorphic to a finitely generated $\mathbb{C}$-algebra. In particular,
    \[\mathcal H(G,\lambda) \cong \begin{cases}
                \mathbb{C}[x, x^{-1}]  & \mbox{if } \mathfrak s \in \mathcal{B}_{\mathrm{cusp}}, \\
			\mathbb{C}[y, z, y^{-1}, z^{-1} ]  & \mbox{if } \mathfrak s \in \mathcal{B}_{\mathrm{neqv}}, \\
                \widetilde{\mathbb{C}}[s, t, t^{-1}]/ \left\langle   s^2-1, t^2 s-s t^2 \right\rangle  & \mbox{if } \mathfrak s \in \mathcal{B}_{\mathrm{eqv}}.  
    \end{cases}\]
\end{proposition}
\begin{proof}
The pair $(K,\lambda)$ is the $\mathfrak s$-type in Theorem \ref{SS}, and therefore 
\[\mathcal H(G,\lambda)\cong \bigotimes\limits_{i=1}^l\,\mathcal H(r_i,q_F^{\mathfrak{f}_i}), \text{ for } r_i = \# \mathcal{I}_i \text{ and } \mathfrak{f}_i = n(\sigma_j) \cdot s(\sigma_j).\]
According to the nature of inertial class $\mathfrak{s}$ (see Eq. \ref{break}), it is convenient to break up the proof into the following three exhaustive cases:

{\bf Case I}: Suppose the inertial class $\mathfrak s$ lies in $\mathcal{B}_{\mathrm{cusp}}$. Then Proposition \ref{prop: cusp} gives the following isomorphism
\[\mathcal H(G,\lambda) \cong \mathbb{C}[x,x^{-1}].\]

{\bf Case II}: Suppose, the inertial class $\mathfrak s$ lies in $\mathcal{B}_{\mathrm{neqv}}$. In this case, the inertial class $\mathfrak s=[L, \sigma]_G$ consists of the diagonal subgroup $L\cong \mathcal{D}^\times \times \mathcal{D}^\times$ of $G$, and a supercuspidal representation $\sigma \simeq \sigma_1 \otimes \sigma_2$ of the diagonal subgroup $L$, where $\sigma_i$'s are irreducible smooth (supercuspidal) representations of $\mathcal{D}^\times$ such that $\sigma_2 $ is not inertially equivalent to $\sigma_1$. Then, we have the following: $k=2$, $n_1=1=n_2$, $\mathcal{I}_{\mathfrak{s}}=\{1,2\}$, and $1\not\sim 2$ in $\mathcal{I}_{\mathfrak{s}}$. Therefore, $\ell=2$, and $r_1=\# \mathcal{I}_1=1=\# \mathcal{I}_2=r_2.$ Hence we get
\[\mathcal H(G,\lambda)=\mathcal H(1,q_F^{\mathfrak{f}_1})\otimes \mathcal H(1,q_F^{\mathfrak{f}_2}),\]
where $\mathfrak{f}_1=n(\sigma_1) s(\sigma_1)$ and $ \mathfrak{f}_2= n(\sigma_2) s(\sigma_2)$ are two positive real numbers. Therefore, using isomorphism (\ref{n:equal:1}) we have
\[\mathcal H(G,\lambda) \cong \mathcal H(1,q_F^{\mathfrak{f}_1})\otimes_{\mathbb{C}} \mathcal H(1,q_F^{\mathfrak{f}_2}) \cong \mathbb{C}[y,y^{-1}] \otimes_\mathbb{C} \mathbb{C}[z,z^{-1}] \cong \mathbb{C}[y,z,y^{-1},z^{-1}].\]

{\bf Case III}: Suppose, the inertial class $\mathfrak s$ lies in $\mathcal{B}_{\mathrm{eqv}}$. In this situation, we have $L = \mathcal{D}^\times \times \mathcal{D}^\times$ and the cuspidal representation $\sigma = \sigma_1 \otimes \sigma_2$ with $\sigma_2$ inertially equivalent to $\sigma_1$. Then $k=2$, $n_1=1=n_2$, $\mathcal{I}_{\mathfrak{s}}=\{1,2\}$, and $1 \sim 2$ in $\mathcal{I}_{\mathfrak{s}}$. Therefore, $\ell=1$, and $r_1=\# \mathcal{I}_1=2$. Hence using Theorem \ref{SS}, we get
\[\mathcal H(G,\lambda) \cong \mathcal H(2,q_F^{\mathfrak{f}_1}),\]
where $\mathfrak{f}_1=n(\sigma_1) s(\sigma_1)$ is a positive real number, which implies $q_F^{\mathfrak{f}_1} \neq -1$. Therefore, Lemma \ref{main:lemma} says $\mathcal H(2,q_F^{\mathfrak{f}_1})  \cong \widetilde{\mathbb{C}}[s, t, t^{-1}]/\left\langle   s^2-1, t^2 s-s t^2 \right\rangle$, and so we have the algebra isomorphism below:
\[\mathcal H(G,\lambda) \cong \widetilde{\mathbb{C}}[s, t, t^{-1}]/ \left\langle s^2-1, t^2 s-s t^2 \right\rangle.\]
    
\end{proof}

\begin{theorem}\label{main:theorem}
  Let $F$ be a locally compact non-archimedean local field and $\mathcal{D}$ be a central division algebra over $F$. Then, the Hecke algebra $\mathcal{H}\left(\mathrm{GL}_2(\mathcal{D})\right)$ is Morita equivalent to the $\mathbb{C}$-algebra
  \[\bigoplus_n \left(\mathbb{C}[x, x^{-1}] \oplus \mathbb{C}[y, z, y^{-1}, z^{-1} ] \oplus \frac{\widetilde{\mathbb{C}}[s, t, t^{-1}]}{\left\langle   s^2-1, t^2 s-s t^2 \right\rangle} \right), \]where $n$ runs over a countable infinite set.
\end{theorem}
\begin{proof}
Using Lemma \ref{Morita:equivalence} and Proposition \ref{decompose:complex-algebra} successively for each inertial class $\mathfrak{s} \in \mathcal{B}(\mathrm{GL}_2(\mathcal{D}))$, we have the following Morita equivalence:
%there exists $\mathfrak s$-type $(K,\lambda)$ together with Proposition \ref{decompose:complex-algebra}, we have
\begin{align*}
    \mathcal{H}\left(\mathrm{GL}_2(\mathcal{D})\right)^\mathfrak s &\sim_M \mathcal{H}(\mathrm{GL}_2(\mathcal{D}), \lambda) \quad \text{for some } \mathfrak s \text{-type } (K, \lambda) \text{ as in Theorem } \ref{SS}\\
    &\simeq \begin{cases}
       \vspace{2 mm}
			\mathbb{C}[x, x^{-1}]  & \mbox{if } \mathfrak s \in \mathcal{B}_{\mathrm{cusp}}, \\
                \vspace{2 mm}
			\mathbb{C}[y, z, y^{-1}, z^{-1} ]  & \mbox{if } \mathfrak s \in \mathcal{B}_{\mathrm{neqv}}, \\
               \widetilde{\mathbb{C}}[s, t, t^{-1}]/ \left\langle   s^2-1, t^2 s-s t^2 \right\rangle  & \mbox{if } \mathfrak s \in \mathcal{B}_{\mathrm{eqv}}. 
    \end{cases}
\end{align*}
The Bernstein spectrum $\mathcal{B}(G)$ is a countably infinite set (cf. \cite[Corolaire 3.9]{Ber}). Moreover, there are countably infinitely many inertial classes $\mathfrak s$ in each $\mathcal{B}_{\mathrm{cusp}}, \mathcal{B}_{\mathrm{neqv}}$, and $\mathcal{B}_{\mathrm{eqv}}$. Therefore, the decomposition (\ref{Hecke:decomposition}) 
\begin{align*}
   \mathcal{H}\left(\mathrm{GL}_2(\mathcal{D})\right) &= \bigoplus\limits_{\mathfrak s \in \mathcal{B}(G)} \mathcal{H}\left(\mathrm{GL}_2(\mathcal{D})\right)^\mathfrak s \\
   &= \bigoplus\limits_{\mathfrak s \in \mathcal{B}_{\mathrm{cusp}}} \mathcal{H}\left(\mathrm{GL}_2(\mathcal{D})\right)^\mathfrak s ~ \bigoplus\limits_{\mathfrak s \in \mathcal{B}_{\mathrm{neqv}}} \mathcal{H}\left(\mathrm{GL}_2(\mathcal{D})\right)^\mathfrak s ~ \bigoplus\limits_{\mathfrak s \in \mathcal{B}_{\mathrm{eqv}}} \mathcal{H}\left(\mathrm{GL}_2(\mathcal{D})\right)^\mathfrak s
\end{align*}
induces a Morita equivalence of the Hecke algebra $\mathcal{H}\left(\mathrm{GL}_2(\mathcal{D})\right)$ and the $\mathbb{C}$-algebra 
\begin{equation*}
\bigoplus_\nu \left( \mathbb{C}[x, x^{-1}] \oplus \mathbb{C}[y, z, y^{-1}, z^{-1} ] \oplus \widetilde{\mathbb{C}}[s, t, t^{-1}]/ \left\langle   s^2-1, t^2 s-s t^2 \right\rangle \right),    
\end{equation*}
 where $\nu$ runs over a countably infinite set.
\end{proof}

\begin{corollary}\label{main:corollary}
Let $\mathcal{D}$ be a central division algebra defined over a locally compact non-archimedean local field. Then, the category of all smooth complex representations of $\mathrm{GL}_2(\mathcal{D})$ does not depend on the division algebra $\mathcal{D}$.
\end{corollary}
\begin{proof}
  Fix two central division algebras $\mathcal{D}_1 $ and $\mathcal{D}_2$ defined over some locally compact non-archimedean local fields. Theorem \ref{main:theorem} implies that the algebra $\mathcal{H}\left(\mathrm{GL}_2(\mathcal{D}_1)\right)$ is Morita equivalent to the algebra $\mathcal{H}\left(\mathrm{GL}_2(\mathcal{D}_2)\right)$. Then, there is an equivalence of categories $\mathcal{H}\left(\mathrm{GL}_2(\mathcal{D}_1)\right)$-Mods and $\mathcal{H}\left(\mathrm{GL}_2(\mathcal{D}_2)\right)$-Mods of non-degenerate modules. Therefore, the categorical equivalence (\ref{Eq:categorical}) induces the following equivalence of categories
   \[\mathcal{R}\left(\mathrm{GL}_2(\mathcal{D}_1)\right) \simeq \mathcal{R}\left(\mathrm{GL}_2(\mathcal{D}_2) \right).\]
\end{proof}

%---------------------------------------------Some remarks-----------------------------------------------------------------------------------------

\begin{remark}
   It isn't easy to extend the proof of Main Theorem to $\mathrm{GL}_N(\mathcal{D})$ for $N\geq 3$ because we can't produce an algebra isomorphism (precisely a Morita equivalence) as in Lemma \ref{main:lemma}. In particular, for $N=3$ the affine Hecke algebras $\mathcal{H}(3, \mathfrak z)$ of type $\tilde{A}_2$ must occur in the decomposition of spherical Hecke algebra of $\mathrm{GL}_3(\mathcal{D})$ (see Theorem \ref{SS}) and $\mathcal{H}(3, \mathfrak z)$ is not Morita equivalent to $\mathcal{H}(3, \mathfrak{z}^\prime)$ if $\mathfrak{z} \neq \mathfrak{z}^\prime$ and $\mathfrak{z} \mathfrak{z}^\prime \neq 1$ (see \cite[Theorem 3.1]{Yan}).
\end{remark}

\begin{remark}
  From the simplicity of the result and the existence of types together with the decomposition of the corresponding Hecke algebras into affine algebras, one can hope that the result holds for lower-rank classical groups and their inner forms. 
\end{remark}

%------------------------------------------------End-----------------------------------------------------------------

\end{document}